\documentclass[12pt]{amsart}

\usepackage{fullpage,amsmath, amssymb,xr}

\newcommand{\bfx}{\mathbf x}
\newcommand{\rs}{\upharpoonright}

\newcommand{\ubM}{M_{\leq 1}}
\newcommand{\ub}[1]{(#1)_{\leq 1}}

\newcommand{\bbN}{{\mathbb N}}

\newcommand{\bbR}{{\mathbb R}}
\newcommand{\bbT}{{\mathbb T}}
\newcommand{\bbC}{\mathbb C}

\newcommand{\cZ}{{\mathcal Z}}

\newcommand{\cB}{\mathcal B}

\newcommand{\e}{\varepsilon}

\newtheorem{thm}{Theorem}[section]
\newtheorem{theorem}[thm]{Theorem}
\newtheorem{corollary}[thm]{Corollary}

\newtheorem{lemma}[thm]{Lemma}
\newtheorem{proposition}[thm]{Proposition}
\newtheorem{prop}[thm]{Proposition}

\theoremstyle{definition}

\newtheorem{definition}[thm]{Definition}
\newtheorem{properties}[thm]{Properties}
\newtheorem*{convention}{Convention}

\newcounter{my_enumerate_counter}
\newcommand{\pushcounter}{\setcounter{my_enumerate_counter}{\value{enumi}}}
\newcommand{\popcounter}{\setcounter{enumi}{\value{my_enumerate_counter}}}

 \DeclareMathOperator{\supp}{supp}

\newcommand{\cU}{{\mathcal U}}
\newcommand{\cV}{\mathcal V}

\newcommand{\bfa}{\mathbf a}
\newcommand{\bfb}{\mathbf b}
\newcommand{\bfc}{\mathbf c}

\newcommand{\NneN}{\bbN^{\nearrow\bbN}}

\newcommand{\fc}{\mathfrak c}

\def\CU{{\mathcal U}}

\def\CA{{\mathcal A}}

\title{Model theory of operator algebras I: Stability}

\author{Ilijas Farah}

\address{Department of Mathematics and Statistics\\
York University\\
4700 Keele Street\\
North York, Ontario\\ Canada, M3J
1P3\\
and Matematicki Institut, Kneza Mihaila 34, Belgrade, Serbia}

\email{ifarah@mathstat.yorku.ca}
\urladdr{http://www.math.yorku.ca/$\sim$ifarah}

\author{Bradd Hart}
\address{Dept. of Mathematics and Statistics\\
McMaster University\\ 1280 Main Street\\ West Hamilton, Ontario\\
Canada L8S 4K1}

\email{hartb@mcmaster.ca}

\urladdr{http://www.math.mcmaster.ca/$\sim$bradd/}
\thanks{The first two authors are partially supported by NSERC. \\ AMS subject codes: 03C45, 03C98, 46L05, 46L10, 46M07}

\author{David Sherman}

\address{Department of Mathematics\\
University of Virginia\\
P. O. Box 400137\\
Charlottesville, VA 22904-4137} \email{dsherman@virginia.edu}

\urladdr{http://people.virginia.edu/$\sim$des5e/}

\begin{document}

\begin{abstract}
Several authors have considered whether the ultrapower and the
relative commutant of a C*-algebra or II$_1$ factor depend on the
choice of the ultrafilter. We settle each of these questions,
extending results of Ge--Hadwin and the first author.
\end{abstract}

\maketitle


\section{Introduction}

Suppose that $A$ is a separable object of some kind: a C*-algebra,
II$_1$-factor or a metric group.  In various places, it is asked
whether all ultrapowers of $A$ associated with nonprincipal
ultrafilters on $\bbN$, or all relative commutants of $A$ in an
ultrapower, are isomorphic. For instance, McDuff (\cite[Question (i)
on p. 460]{McDuff:Central}) asked whether all relative commutants of
a fixed separable II$_1$ factor in its ultrapowers associated with
nonprincipal ultrafilters on~$\bbN$ are isomorphic.  An analogous
question for relative commutants of separable C*-algebras was asked
by Kirchberg and answered in \cite{Fa:Relative} for certain
C*-algebras. (The argument in \cite{Fa:Relative} does not cover all
cases of real rank zero as stated.) As a partial answer to both
questions, Ge and Hadwin (\cite{GeHa}) proved that the Continuum
Hypothesis implies a positive answer. They also proved that if the
Continuum Hypothesis fails then some C*-algebras have nonisomorphic
ultrapowers associated with nonprincipal
 ultrafilters on $\bbN$.
We give complete answers to all these questions in Theorems
\ref{T.type-II-1}, \ref{T.rel.comm} and \ref{T.C*}.  

During the December 2008 Canadian Mathematical Society meeting in
Ottawa,  Sorin Popa asked the first author whether one can find
uncountably many non-isomorphic tracial ultraproducts of finite
dimensional matrix algebras
 $M_i(\bbC)$, for $i\in\bbN$. In Proposition~\ref{P.A1} we
 show the if the Continuum Hypothesis fails then there are nonisomorphic
 ultraproducts and if the continuum is sufficiently large then Popa's question has a positive answer.

These results can be contrasted with the fact that all ultrapowers
of a separable Hilbert space (or even a Hilbert space of character
density $\leq \fc=2^{\aleph_0}$) associated with nonprincipal
ultrafilters on $\bbN$ are isomorphic to $\ell^2(\fc)$. We show that
the separable tracial von Neumann algebras whose ultrapowers are all
isomorphic even when the Continuum Hypothesis fails are exactly
those of type I (Theorem \ref{T:typeI}).

We now introduce some terminology for operator algebraic ultrapowers
that we will use throughout the paper. {\bf Unless we say otherwise, all ultrafilters we use in this paper are non-principal ultrafilters on $\bbN$.}

A von Neumann algebra $M$ is \textit{tracial} if it is equipped with
a faithful normal tracial state $\tau$.  A finite factor has a
unique tracial state which is automatically normal.  The metric
induced by the $\ell^2$-norm, $\|a\|_2=\sqrt{\tau(a^*a)}$, is not
complete on $M$, but it is complete on the unit ball (in the
operator norm).  The completion of  $M$ with respect to this metric
is
 isomorphic to a Hilbert space (see e.g.,
\cite{Black:Operator}~or~\cite{Jon:von}).

The algebra of all sequences in $M$ bounded in the operator norm is
denoted by $\ell^\infty(M)$. If~$\cU$ is an ultrafilter on $\bbN$
then
\[
\textstyle c_{\cU}=\{\vec a\in \ell^\infty(M): \lim_{i\to \cU}
\|a_i\|_2=0\}
\]
is a norm-closed two-sided ideal in $\ell^\infty(M)$, and the
\emph{tracial ultrapower} $M^{\cU}$ (also denoted by $\prod_{\cU}
M$) is defined to be the quotient  $\ell^\infty(M)/c_{\cU}$.  It is
well-known that $M^{\cU}$ is tracial, and a factor if and only if
$M$ is---see e.g., \cite{Black:Operator} or \cite{Tak:TheoryIII}. In
the sequel to this paper (\cite{FaHaSh:Model2}) we shall demonstrate
that  this follows from axiomatizability in first order continuous
logic of tracial von Neumann algebras and the Fundamental Theorem of
Ultraproducts.

The elements of $M^{\cU}$ will be denoted by boldface Roman letters
such as $\bfa$, $\bfb$, $\bfc$,\dots{} and their representing
sequences in $\ell^\infty(M)$ will be denoted by $a(i)$, $b(i)$,
$c(i)$, \dots, for $i\in \bbN$, respectively.  Identifying a tracial
von Neumann algebra $M$ with its diagonal image in  $M^{\cU}$, we
will also work with the \emph{relative commutant} of $M$ in its
ultrapower,
\[
M'\cap M^{\cU}=\{\bfb: (\forall a\in M) a\bfb=\bfb a\}.
\]
A brief history of tracial ultrapowers of II$_1$ factors can be
found in the introduction to \cite{She:Notes}.

We will use several variations of the ultraproduct construction.
For C*-algebras, $c_{\cU}$ consists of sequences which go to zero in
the operator norm.  For groups with bi-invariant metric, $c_{\cU}$
is the normal subgroup of sequences whose ultralimit along $\cU$ is the
identity (\cite{Pe:Hyperlinear}).  One may also form the
\emph{ultraproduct} of a sequence of distinct algebras or groups; in
fact all of these are special cases of the ultraproduct construction
from the model theory of metric structures (see
\cite{FaHaSh:Model2},   \cite{BYBHU}, or \cite{GeHa}).


 The methods used in the present paper make no explicit use of logic, and a reader can understand all proofs while
 completely ignoring all references to logic.
However,  the intuition coming from model theory was indispensable
in discovering our results. With future applications in mind, we
shall outline the model-theoretic framework for the study of
operator algebras in \cite{FaHaSh:Model2}.

\section{The order property} \label{S.OA}

Our results  are stated and proved for tracial von Neumann algebras.
Later on we shall point out that the arguments work in more general
contexts, including those of C*-algebras and unitary groups.

\label{S.Setup} Let $M$ be a tracial von Neumann algebra  and let
$\ubM$ denote its unit ball with respect to the operator norm. For
$n\geq 1$ and a
*-polynomial $P(x_1,\dots, x_n, y_1,\dots, y_n)$ in $2n$ variables
consider the function
\[
g(\vec x,\vec y)=\|P(\vec x,\vec y)\|_2
\]
on the $2n$-th power of $\ubM$. In this section all functions of this kind range
over the unit ball of $M$. Note that with this convention we have
the following:

\begin{properties}\label{Properties.g}
\begin{enumerate}
\item [(G1)] $g$ defines a uniformly continuous function on the $2n$-th power of the
unit ball of any tracial von Neumann algebra. The uniform continuity does not depend on the particular algebra; that is, for every $\epsilon$ there is a $\delta$ independent of the choice of algebra;
\item [(G2)] For every ultrafilter $\cU$, the function $g$ can be canonically extended to
the $2n$-th power of the unit ball of the ultrapower
\end{enumerate}
\end{properties}

The discussion provided below applies verbatim to any $g$ satisfying
Properties~\ref{Properties.g}. Readers familiar with model theory will
notice that if $g$ is an interpretation in continuous logic of a
$2n$-ary formula, then Properties~\ref{Properties.g} are satisfied
(see \cite{FaHaSh:Model2} for definition of a formula and properties
of an interpretation of
  a formula). We shall furthermore suppress the mention of $n$
  whenever it is irrelevant.
\begin{convention}
In the remainder of this section we refer to $g$ and $n$ that
satisfy Properties~\ref{Properties.g} as `a $2n$-ary formula.'
\end{convention}
Each~$g$ used in our applications will be of the form $\|P(\vec
x,\vec y)\|_2$ for some *-polynomial~$P$.

For $0\leq \e<1/2$
 define the relation $\prec_{g,\e}$ on $(\ubM)^n$ by
$$
\vec x_1\prec_{g,\e} \vec x_2 \mbox{ if } g(\vec x_1,\vec x_2)\leq
\e \mbox{ and } g(\vec x_2,\vec x_1)\geq 1-\e
$$
Note that we do not require that $\prec_{g,\e}$ defines an ordering
on its domain. However, if $\vec x_i$, for $1\leq i\leq k$, are such
that $\vec x_i\prec_{g,\e} \vec x_j$ for $i<j$ then these $n$-tuples
form a linearly ordered chain of length $k$. We call such a
configuration a \emph{$g$-$\e$-chain of length $k$ in $M$}.

We write $\prec_g$ for $\prec_{g,0}$. The following is a special
case of \L os' theorem for ultraproducts in the logic of metric
structures (see \cite[Theorem~5.4]{BYBHU} or
\cite[Proposition 4.3]{FaHaSh:Model2}).

\begin{lemma}\label{L0}
For  a formula $g$, an ultrafilter $\cU$, and $\bfa$ and $\bfb$ in
$\prod_i \ub{M_i}$ the following are equivalent.
\begin{enumerate}
\item For every $\e>0$ we have $\{i: a(i)\prec_{g,\e} b(i)\}\in \cU$,
\item $\bfa\prec_g \bfb$.
\end{enumerate}
\end{lemma}


The following special case of the abstract definition of the order
property in a model
(see  \cite[Section 7]{BYU:ContStab} or \cite[Definition 5.2]{FaHaSh:Model2}) is modeled on the
order property in stability theory (\cite{She:Classification}).

\begin{definition} \label{Def.vNA.OP}
A tracial von Neumann algebra $M$ has the \emph{order property} if
there exists a formula $g$ such that for every $\e>0$, $M$ has
arbitrarily long finite $g$-$\e$-chains. If we wish to make the $g$ explicit, we say that $M$ has the order property with respect to $g$.

The following terminology is non-standard but convenient for what
follows: A sequence $M_i$, for $i\in \bbN$, of tracial von Neumann
algebras has the \emph{order property with respect to $g$} if  for every $\e>0$ and every $k\in \bbN$ all but finitely
many of the $M_i$, for $i\in \bbN$, have a $g$-$\e$-chain of length
$k$.  We say that the sequence of $M_i$'s has the order property if it has the order property with respect to some $g$.
\end{definition}

The analysis of gaps in quotient structures is behind a number of
applications of set theory to functional analysis (see e.g.,
\cite{DaWo:Introduction}, \cite{Fa:All}).
 Let $\lambda$ be a regular cardinal. An
\emph{$(\aleph_0,\lambda)$-$g$-pregap} in $M$  is a pair consisting
of a $\prec_g$-increasing family $\bfa_m$, for $m\in \bbN$, and a
$\prec_g$-decreasing family $\bfb_\gamma$, for $\gamma<\lambda$,
such that $\bfa_m\prec_g \bfb_\gamma$ for all $m$ and $\gamma$. A
$\bfc$ such that $\bfa_m\prec_g\bfc$ for all $m$ and
$\bfc\prec_g\bfb_\gamma$ for all $\gamma$ is said to \emph{fill} (or
\emph{separate}) the pregap. An $(\aleph_0,\lambda)$-pregap that is
not separated is an \emph{$(\aleph_0,\lambda)$-gap}.

Assume $M_i$, $i\in\bbN$, are tracial von Neumann algebras. Assume
$\cU$ is a nonprincipal ultrafilter on $\bbN$ such that for every
$m\geq 1$ the set of all $i$ such that $M_i$ has a $g$-$1/m$-chain
of length $m$ belongs to $\cU$. Then we can find sets $Y_m\in \cU$
such that $Y_m\supseteq Y_{m+1}$, $\bigcap_m Y_m=\emptyset$ and for
every $i\in Y_m\setminus Y_{m+1}$ there exists a $g$-$1/m$-chain
$\vec a_0(i), \vec a_1(i)$,\dots, $\vec a_{m-1}(i)$ in $M_i$.
Letting $Y_0=\bbN$ and using  $\bigcap_m Y_m=\emptyset$ for each
$i\in \bbN$ we define $m(i)$ as the unique $m$ such that $i\in
Y_m\setminus Y_{m+1}$.

 For
$h\in\bbN^{\bbN}$ define $\vec a_h$ in
 $\prod_{i\in \bbN} M_i^n$ by $\vec a_h(i)=\vec a_{h(i)}(i)$ if $h(i)\leq m(i)-1$
 and $\vec a_h(i)=\vec a_{m(i)-1}(i)$ otherwise. Then  $\bf a_h$ denotes the element of $\prod_{\cU} M_i^n$
 with the representing sequence $\vec a_h$. Write $\bar m$ for the constant function $\bar m(i)=m$.
Let  $\bfa_m$ denote~$\bfa_{\bar m}$. By $\NneN$ we denote the set
of all nondecreasing functions $f\colon \bbN\to \bbN$ such that
$\lim_n f(n)=\infty$, ordered pointwise.

\begin{lemma}\label{L2} With the notation as in the above paragraph,
assume $\bfb\in \prod_{\cU} M_i^n$ is such that $\bfa_m\prec_g \bfb$
for all $m$. Then there is $h\in \NneN$ such that  $\bfa_h\prec_g
\bfb$ and  $\bfa_m\prec_g \bfa_h$ for all $m$.
\end{lemma}

\begin{proof} For $m\in \bbN$ let
\[
X_m=\{i\in Y_m: (\forall k\leq m) g(a_k(i),b(i))\leq 1/m\text{ and }
g(b(i),a_k(i))\geq 1-1/m\}.
\]
 Clearly $X_m\in \cU$ and $\bigcap_m
X_m=\emptyset$. For $i\in X_m\setminus X_{m+1}$ let $h(i)=m$ and let
$h(i)=0$ for $i\notin X_0$. Then for  $m<m'$ and $i\in X_{m'}$ we
have $h(i)\geq m'\geq m$ hence  $g(a_m(i),a_{h(i)}(i))\leq 1/m'$ and
 $g(\bfa_m,\bfa_h)=0$. Similarly $g(\bfa_h,\bfa_m)=1$ and
therefore $\lim_{i\to \cU} h(i)=\infty$.

 Also, if $i\in X_m$ then
$g(a_{h(i)},b(i))\leq 1/m$ and therefore $g(\bfa_h,\bfb)=0$.
Similarly $g(\bfb,\bfa_h)=1$ and the conclusion follows.
\end{proof}

Following \cite{Do:Ultrapowers} (see also \cite{Fa:Relative}), for
an ultrafilter $\cU$ we write $\kappa(\cU)$ for the
\emph{coinitiality} of $\NneN/\cU$, i.e., the minimal cardinality of
$X\subseteq \NneN$ such that for every $g\in \NneN$ there is $f\in
X$ such that $\{n\colon f(n)\leq g(n)\}\in \cU$. (It is not
difficult to see that this is equal to $\kappa(\cU)$ as defined in
 \cite[Definition~1.3]{Do:Ultrapowers}.)

\begin{lemma}\label{L3} Assume $M_i$, $i\in\bbN$, is a sequence of
tracial von Neumann algebras with the order property with respect to $g$ and $\cU$ is a
nonprincipal ultrafilter on $\bbN$.
 Then $\kappa(\cU)$ is equal to the minimal cardinal $\lambda$ such that
$\prod_{\cU} M_i$ contains an $(\aleph_0,\lambda)$-$g$-gap.
\end{lemma}

\begin{proof}

The proof is similar to the  last paragraph of the proof of
\cite[Proposition~6]{Fa:Relative}. 
Let $Y_m$ and $\vec a_0(i),\vec a_1(i),\dots, \vec a_{m-1}(i)$ be
as in the paragraph before Lemma~\ref{L2} and we shall use the
notation $\bfa_h$ and $\bfa_m$ as introduced there.

Fix functions $h(\gamma)$, $\gamma<\kappa(\cU)$, which  together
with the constant functions $\bar m$, for $m\in \bbN$ form a gap. We
claim that $\bfa_{\bar m}$, for $m\in \bbN$, form a gap with
$\bfb_\gamma=\bfa_{h(\gamma)}$, for $\gamma<\kappa(\cU)$. It is
clear that these  elements form a
$(\aleph_0,\kappa(\cU))$-$g$-pregap and by Lemma~\ref{L2} this
pregap is not separated.

Now assume $\lambda$ is the minimal cardinal such that $\prod_{\cU}
M_i$ contains an $(\aleph_0,\lambda)$-$g$-gap, and let $\bfa_m$, for $m\in \bbN$,
$\bfb_\gamma$, for $\gamma<\lambda$, be an
$(\aleph_0,\lambda)$-$g$-gap in $\prod_{\cU}M_i$. Fix a representing
sequence $a_m(i)$, for $i\in \bbN$, of $\bfa_m$. For each $m$ the
set
\[
X_m=\{i \in Y_m: i \geq m \mbox{ and } a_0(i),a_1(i),\dots, a_{m-1}(i)\text{ form a
$g$-$1/m$-chain}\}
\]
belongs to $\cU$ and $X_m\supseteq X_{m+1}$ for all $m$. As before,
for $h\in \bbN$ define $\bfa_h$ via its representing sequence. Let
$a_h(i)=a_{h(i)}(i)$ if  $i\in X_{h(i)}$ and $a_h(i)=a_m(i)$ if
$i\notin X_{h(i)}$ and $m$ is the maximal such that $i\in X_m$. For
$i\notin X_0$ define $a_h(i)$ arbitrarily. Note that
 $\lim_{i\to \cU} h(i)=\infty$ if and only if  $\bfa_m\prec_g \bfa_h$.

By Lemma~\ref{L2} for every $\gamma<\lambda$ we can find $h(\gamma)$
such that $\bfb'_\gamma=\bfa_{h(\gamma)}$ is $\prec_g \bfb_\gamma$
and $\bfa_m\prec_g\bfb'_\gamma$ for all $m$. In addition we choose
$h(\gamma)$ so that $h(\gamma)\leq^*h(\gamma')$ for all
$\gamma'<\gamma$. This is possible by the minimality of $\lambda$.
The functions $h(\gamma)$, for $\gamma<\lambda$, together with the
constant functions form an $(\aleph_0,\lambda)$-gap in
$\bbN^\bbN/\cU$.
\end{proof}

\begin{proposition}\label{P3} Assume the Continuum Hypothesis fails.
Assume $M_i$, for $i\in \bbN$, is a sequence of tracial von Neumann
algebras
 with the order property. Then there exist ultrafilters $\cU$
and $\cV$ such that the ultraproducts $\prod_i M_i/c_{\cU}$ and
$\prod_i M_i/c_{\cV}$ are not isomorphic.
\end{proposition}

\begin{proof}
Fix $n$ and a $2n$-ary formula  $g$  such that
 for every $m\in \bbN$ the set of all $i$ such
that there is a $g$-$1/m$-chain of length $\geq m$ in
 $M_i^n$ is cofinite.

Let $X\subseteq \bbN$ be an infinite set such that for all $m\in
\bbN$ there is a $g$-$1/m$-chain of length $m$ in the unit ball of
$M_i^n$ for all but finitely many $i$ in $X$. By
\cite[Theorem~2.2]{Do:Ultrapowers} (also proved by Shelah,
\cite{She:Classification}) there are  $\cU$ and $\cV$ so that
$\kappa(\cU)=\aleph_1$ and $\kappa(\cV)=\aleph_2$ (here $\aleph_1$
and $\aleph_2$ are the least two uncountable cardinals; all that
matters for us is that they are both $\leq 2^{\aleph_0}$ and
different).

 By Lemma~\ref{L3} the ultraproduct associated with $\cU$ has
$(\aleph_0,\aleph_1)$-$g$-gaps 
and the ultraproduct
associated with $\cV$ has $(\aleph_0,\aleph_2)$-$g$-gaps but no
$(\aleph_0,\aleph_1)$-$g$-gaps.
Therefore these ultraproducts are not isomorphic.
\end{proof}

Our definition of the order property for relative commutants is a
bit more restrictive than that of the order property for tracial von
Neumann algebras. 
We say that $a$ and $b$ in a tracial von Neumann algebra
\emph{$1/m$-commute} if $\|[a,b]\|_2<1/m$.

\begin{definition}\label{Def.vNA.rcOP}
If~$M$ is a tracial von Neumann algebra then we say that the
\emph{relative commutant type of $M$ has the order property with
respect to $g$} if there are $n$ and  a *-polynomial $P(x_1,\dots,
x_n, y_1,\dots, y_n)$ in $2n$ variables such that with $g(\vec
x,\vec y)=\|P(\vec x,\vec y)\|_2$,
  for  every finite $F\subseteq M$,  
  and every $m\in \bbN$,  there is a
$g$-$1/m$-chain of length $m$ in $M_{\leq 1}^n$ all of whose
elements $1/m$-commute with all elements of~$F$.
\end{definition}

Another remark for model-theorists is in order. In this definition
we could have allowed $g$ to be an arbitrary atomic formula, but we don't have
an application for the more general definition.

\begin{lemma} \label{L4} Assume $M$ is a separable tracial von Neumann algebra
and $\cU$ is a nonprincipal ultrafilter on $\bbN$.  If the relative
commutant type of $M$ has the order property with respect to
$g$, then for every uncountable regular cardinal $\lambda$ the
following are equivalent.
\begin{enumerate}
\item \label{L4.1} $\kappa(\cU)=\lambda$.
\item \label{L4.2} The relative commutant $M'\cap M^{\cU}$ contains an
$(\aleph_0,\lambda)$-$g$-gap.
\end{enumerate}
Note that $g$ and $\prec_g$ are isomorphism invariants for relative
commutants.
\end{lemma}

\begin{proof} The proof is very similar to the proof of
Lemma~\ref{L3} and we add just a few clarifying remarks. Returning to the paragraphs before
Lemma~\ref{L2}, let us define the sets $Y_m$ in the context of the relative commutant.  Fix a countable dense
subset $F$ of our separable $M$ and write $F$ as an increasing sequence of finite sets $F_n$.  Now define
a decreasing sequence of sets $Y_m \in \cU$ such that $\bigcap Y_m = \emptyset$ and
such that for all $i \in Y_m$, there is a $g$-$1/m$-chain $\vec a_0(i), \vec a_1(i)$,\dots, $\vec a_{m-1}(i)$
made up of elements of $M$ which $1/m$-commute with all the elements of $F_m$.
The point of the exercise is that if we now define $\vec a_h$ as before, the resulting element of the ultraproduct will be in the relative commutant.
Lemma~\ref{L2} can now be reformulated so that if $\bfb$ is in the relative commutant then so is the constructed $\bfa_h$.  Looking at the proof of 
Lemma~\ref{L3}, the proof may be copied verbatim simply replacing each mention of the ultraproduct with the relative commutant and noting that the various elements of the ultraproduct constructed in the proof are now provably in the relative commutant.

 The
fact that $g$ and $\prec_g$ are isomorphism invariants for relative
commutants follows from the fact that the evaluation of $g$ depends
only on the elements of the relative commutant (i.e., $g$ is a
quantifier-free formula).
\end{proof}

Using Lemma~\ref{L4} instead of Lemma~\ref{L3} the proof of
Proposition~\ref{P3} gives the following.

\begin{proposition} \label{P4} Assume the Continuum Hypothesis fails.
If $M$ is a separable tracial von Neumann algebra whose relative
commutant type has the order property  then there are nonprincipal
ultrafilters $\cU$ and $\cV$ on $\bbN$ such that the relative
commutants $M'\cap M^{\cU}$ and $M'\cap M^{\cV}$ are not isomorphic.
\qed
\end{proposition}

The preceding proofs used very little of the assumption that we were
working with tracial von Neumann algebras  and they apply to the
context of C*-algebras, as well as unitary groups of tracial von
Neumann algebras  or C*-algebras.




The following proposition
 can be proved by copying the proofs of Proposition~\ref{P3} and
Proposition~\ref{P4} verbatim.


\begin{proposition} \label{P4.C*} Assume the Continuum Hypothesis fails, and let $A$ be a separable C*-algebra  or metric group.  If $A$ has the order
property then there are nonprincipal ultrafilters $\cU$ and $\cV$ on
$\bbN$ such that $A^{\cU}$ and $A^{\cV}$ are not isomorphic. If the
relative commutant type of $A$ has the order property then there are
nonprincipal ultrafilters $\cU$ and $\cV$ on $\bbN$ such that the
relative commutants $A'\cap A^{\cU}$ and $A'\cap A^{\cV}$ are not
isomorphic. \qed
\end{proposition}


\section{Type  II$_1$ factors}

The main result of this section is Theorem~\ref{T.type-II-1}. The
case \eqref{T.type-II-1.3} implies \eqref{T.type-II-1.1} was proved
in \cite{GeHa}.

For a II$_1$ factor $M$ let $U(M)$ denote its unitary group. It is a
complete metric group with respect to the $\ell^2$-metric. Hence the
ultrapower of $U(M)$ is defined in the usual manner (see
\cite{Pe:Hyperlinear} and \cite{BYBHU}).

\begin{thm} \label{T.type-II-1} For a
II$_1$ factor $M$ of cardinality $\fc$  the following are
equivalent.
\begin{enumerate}
\item \label{T.type-II-1.1} For all nonprincipal ultrafilters $\cU$ and $\cV$ the
ultrapowers $M^{\cU}$ and $M^{\cV}$ are isomorphic.

\item \label{T.type-II-1.2b}For all nonprincipal ultrafilters $\cU$ and $\cV$ the
 ultrapowers of the unitary groups
$U(M)^{\cU}$ and $U(M)^{\cV}$ are isomorphic.

\item \label{T.type-II-1.3}The Continuum Hypothesis holds.
\end{enumerate}

\end{thm}

In fact
the implications \eqref{T.type-II-1.1} $\Leftrightarrow$  \eqref{T.type-II-1.2b} $\Rightarrow$ \eqref{T.type-II-1.3} are true
for an arbitrary II$_1$ factor.

The proof of this theorem will be given after a sequence of lemmas.

On  a II$_1$ factor $M$  define $g\colon (\ubM)^4\to \bbR$ by
\[
g(a_1,b_1,a_2,b_2)=\|[a_1,b_2]\|_2.
\]
In the following we consider $M_{2^n}(\bbC)$ with respect to its
$\ell_2$-metric. The analogous statements for the operator metric
are true, and easier to prove (see \cite{Fa:Relative}).

\begin{lemma} \label{L.A1}
\begin{enumerate}
\item \label{L.A1.1} The sequence $M_{2^n}(\bbC)$, for $n\in \bbN$, has the order property.
\item \label{L.A1.2} The sequence $M_{n}(\bbC)$, for $n\in \bbN$, has the order property.
\item \label{L.A1.3} Every  II$_1$ factor (and every sequence of II$_1$ factors) has the order property.
\end{enumerate}
\end{lemma}

\begin{proof} \eqref{L.A1.1}
We prove that for every $n$ in $M_{2^n}(\bbC)$ there is a $g$-chain
of length $n-1$. Identify $M_{2^n}(\bbC)$ with
$\bigotimes_{i=0}^{n-1} M_2(\bbC)$. Let $x=\begin{pmatrix} 0 & \sqrt
2 \\ 0 & 0
\end{pmatrix}$ and let $y=x^*=\begin{pmatrix} 0 & 0 \\ \sqrt 2 & 0
\end{pmatrix}$. Then we have $\|x\|_2=1=\|y\|_2$.  Also $[x,y]=\begin{pmatrix} 2 & 0 \\ 0 & -2
\end{pmatrix}$ and $\|[x,y]\|_2=2$.   For $1\leq i\leq n-1$ let
\[
a_i=\bigotimes_{j=0}^i x \otimes \bigotimes_{j=i+1}^{n-1} 1\text{
and }b_i=\bigotimes_{j=0}^i 1 \otimes y
\otimes\bigotimes_{j=i+2}^{n-1} 1.
\]
We have $\|a_i\|_2=\prod_{j=0}^i \|x_j\|_2=1=\|b_i\|_2$ for all $i$.
 Clearly $\|[a_i,b_j]\|_2=0$ if $i\leq j$ and
$\|[a_i,b_j]\|_2=\|[x,y]\|_2=2$ if $i>j$ and therefore pairs
$(a_i,b_i)$, for $1\leq i\leq n-1$, form a $g$-chain.

\eqref{L.A1.2}  We prove that for all  $\e>0$, $n\in \bbN$  and a
large enough $m$ there is a $g,\e$-chain of length $n$ in
$M_m(\bbC)$. Assume $m$ is much larger than $2^n$ so that $m=k\cdot
2^n+r$ with $r/m$ sufficiently small. Then pick a projection $p$ in
$M_m(\bbC)$ with $\tau(p)=k\cdot 2^n$, identify the corner $p
M_m(\bbC)p $ with $M_{2^n}(\bbC)\otimes M_k(\bbC)$ and apply
\eqref{L.A1.1}.

\eqref{L.A1.3} Since every II$_1$ factor $M$ has a unital copy of
$M_{2^n}(\bbC)$ for every $n$ this follows immediately from
\eqref{L.A1.1}.
\end{proof}

The question whether there are non-isomorphic non-atomic tracial
ultraproducts of full matrix algebras  $M_n(\bbC)$ was raised by S.
Popa (personal communication). This question first appeared in
  \cite{Li:Ultraproducts}, after Theorem 3.2.

\begin{prop} \label{P.A1} Assume the Continuum Hypothesis fails. Then there
are nonprincipal ultrafilters $\cU$ and $\cV$ on $\bbN$ such that
the II$_1$ factors $\prod_{\cU} M_n(\bbC)$ and $\prod_{\cV}
M_n(\bbC)$ are nonisomorphic. Moreover, there are at least as many
nonisomorphic ultraproducts as there are uncountable cardinals below
$\fc$.
\end{prop}

\begin{proof}Again this follows by
Proposition~\ref{P3} and Lemma~\ref{L.A1}.
\end{proof}

We don't know whether the Continuum Hypothesis implies that all
ultraproducts of $M_n(\bbC)$, for $n\in \bbN$, associated with
nonprincipal ultrafilters on $\bbN$ are isomorphic. This is
equivalent to asking whether the continuous first order theories of
matrix algebras $M_n(\bbC)$ converge as $n\to \infty$ (see
\cite{FaHaSh:Model2}).

\begin{proof}[Proof of Theorem~\ref{T.type-II-1}.]
Assume  the  Continuum Hypothesis holds. Clause
\eqref{T.type-II-1.2b} is an immediate consequence of clause
 \eqref{T.type-II-1.1}, which
 was proved by Ge and Hadwin in~\cite{GeHa} (note their proof in
  \cite[Section 3]{GeHa}
 only uses the fact that the algebra has cardinality $\fc$).

Now assume the Continuum Hypothesis  fails. Then
\eqref{T.type-II-1.1} follows by Proposition~\ref{P3} and
Lemma~\ref{L.A1}.

It remains to show \eqref{T.type-II-1.2b}.
 This can be proved directly by taking advantage of the commutators (cf. the proof of Lemma~\ref{L.A1})
but instead we show that it already follows from what was proved so
far. It is well-known that $U(M^{\cU})=U(M)^{\cU}$. In
model-theoretic terms (see \cite{BYBHU})  this equality states that
the
  unitary group is definable in a
II$_1$ factor.  The conclusion then follows from Dye's result
(\cite{Dye:On}) that two von Neumann algebras are isomorphic if and
only if their unitary groups
are isomorphic (even as discrete groups). 
\end{proof}

\section{Tracial von Neumann algebras}

We remind the reader that any tracial von Neumann algebra can be
written as  $M_{\text{II}_1} \oplus M_{\text{I}_1} \oplus
M_{\text{I}_2} \dots$, with subscripts the types of the summands.
The goal of this section is to prove that a tracial von Neumann
algebra does not have the order property if and only if it is type I
(Theorem \ref{T:typeI}).

Since most of the literature on tracial ultrapowers has focused on
factors, we start by establishing that this operation commutes with
forming weighted direct sums (Lemma \ref{T:dsum}) and taking the
center (Corollary \ref{T:center}).  Corollary \ref{T:center} follows from Lemma A.4.2 in \cite{SS} but for the reader's convenience we give a short proof of our special case that contains some extra information in Lemma \ref{T:dixmier}. Let $(M, \tau)$ be a tracial von
Neumann algebra, with $\cZ(M)$ its center.  Fix a free ultrafilter
$\CU$ on $\mathbb{N}$.

\begin{lemma} \label{T:dsum}
If $\{z_j\} \subset M$ are nonzero central projections summing to 1,
then
\[
\textstyle(M,\tau)^\CU \simeq \sum^\oplus (z_j M,
\frac{1}{\tau(z_j)} \tau|_{z_j M})^\CU.
\]
\end{lemma}

\begin{proof}
We claim that the map
$$M^\CU \ni(x_i) \mapsto ( (z_j x_i)_i)_j \in \sum\nolimits_j^\oplus (z_j M)^\CU$$
is a well-defined *-isomorphism.

To see that it is well-defined, suppose $(x_i) = (x'_i)$ in $M^\CU$.
Then as $i \to \CU$, $\|x_i - x'_i\|_2 \to 0$.  Thus for any $j$,
$\|z_j (x_i - x'_i)\|_2 \to 0$.  By rescaling the trace this implies
that $(z_j x_i)_i = (z_j x'_i)_i$ in $(z_j M)^\CU$.

It is then clearly a *-homomorphism.

We show next that it is injective.  Suppose that $(x_i)$ belongs to
the kernel.  We may assume that $\sup\|x_i\| \leq 1$.  We have that
for all $j$, $\|z_j x_i\|_2 \to 0$ as $i \to \CU$.  Given $\e > 0$,
let $N$ be such that $\sum_{j>N} \tau(z_j) < \e$.  Then
\[
\textstyle \|x_i\|_2^2 = \sum_j \|z_j x_i\|_2^2 < \sum_{j<N} \|z_j
x_i\|_2^2 + \e,
\]
which is $< 2\e$ for $i$ near $\CU$.  Since $\e$ was arbitrary, $x_i
\to 0$ in $L^2$.

Finally, for surjectivity let $((y^j_i)_i)_j \in \sum^\oplus (z_j
M)^\CU$.  We may assume that $1 = \|((y^j_i)_i)_j\|_\infty = \sup_j
\|(y^j_i)_i\|_\infty$, and then we may also assume that the
representing sequences have been chosen so that $\|y^j_i\|_\infty
\leq 1$ for each $i$ and $j$.  Identifying each $y^j_i$ with its
image in the inclusion $z_j M \subseteq M$, set $x_i = \sum_j
y^j_i$.  Since each set $\{y^j_i\}_j$ consists of centrally
orthogonal elements, $\|x_i\|_\infty = \sup_j \|y^j_i\|_\infty \leq
1$.  Thus $(x_i)$ defines an element of $M^\CU$, and by construction
$(x_i)$ is mapped to the element initially chosen.
\end{proof}

The following lemma may be in the literature, but we do not know a
reference.

\begin{lemma} \label{T:dixmier}
Let $M$ be a finite von Neumann algebra and $T$ its unique
center-valued trace.  For any $x \in M$, we have
$$\|x - T(x)\|_2 \leq \sup_{y \in M_{\leq 1}} \|[x,y]\|_2 \leq 2 \|x - T(x)\|_2.$$
\end{lemma}

\begin{proof}
We will use the Dixmier averaging theorem (\cite{D:anneaux}), which
says in this situation that for any $\e >0$ there is a finite set of
unitaries $\{u_j\} \subset M$ and positive constants $\lambda_j$
adding to 1 such that $\|T(x) - \sum \lambda_j u_j x u_j^*\|_\infty
< \e$.  We only need the $\ell^2$ estimate to compute
\begin{align*}
\|T(x) - x\|_2 &\leq \e + \left\|\sum \lambda_j u_j x u_j^* - x
\right\|_2 \\ &= \e + \left\|\sum \lambda_j [u_j, x] \right\|_2 \\
&\leq \e + \sum \lambda_j \|[u_j, x]\|_2 \\ &\leq \e + \sup_{y \in
M_{\leq 1}} \|[x,y]\|_2.
\end{align*}
Since $\e$ is arbitrary, the first inequality of the lemma follows.

The second inequality is more routine:
$$\sup_{y \in M_{\leq 1}} \|[x,y]\|_2 = \sup_{y \in M_{\leq 1}} \|[x - T(x),y]\|_2 \leq 2 \|x - T(x)\|_2. \qedhere$$
\end{proof}

\begin{corollary} \label{T:center}
For $M$ a tracial von Neumann algebra, $\cZ(M^\CU) \simeq
\cZ(M)^\CU$.
\end{corollary}

\begin{proof}
The inclusion $\supseteq$ is obvious.  For the other inclusion,
suppose $(x_i) \in \cZ(M^\CU)$.  For each $i$, let $y_i \in M_{\leq
1}$ be such that $\|[x_i, y_i]\|_2 \geq \frac12 \sup_{y \in M_{\leq
1}} \|[x_i,y]\|_2$.  By Lemma \ref{T:dixmier}, $\|x_i - T(x_i)\|_2
\leq \sup_{y \in M_{\leq 1}} \|[x_i,y]\|_2 \leq 2\|[x_i, y_i]\|_2$.
Centrality of $(x_i)$ means that the last term goes to 0 as $i \to
\CU$, so the first term does as well, and $(x_i) = (T(x_i))$ is
central in $M^\CU$.
\end{proof}

\begin{lemma} \label{T:typeIn}
If $M$ is type $\text{I}_n$, then $M^\CU$ is type $\text{I}_n$.
More specifically, $M \simeq \mathbb{M}_n \otimes \cZ(M)$ and $M^\CU
\simeq \mathbb{M}_n \otimes \cZ(M)^\CU$.
\end{lemma}

\begin{proof}
A direct calculation shows that an abelian projection in $M$ is
carried under the diagonal embedding $M \hookrightarrow M^\CU$ to an
abelian projection in $M^\CU$.

An algebra is type $\text{I}_n$ if and only if it contains $n$
equivalent abelian projections summing to 1.  If $M$ has such
projections, their images in $M^\CU$ have the same properties.

The second sentence of the lemma follows from Corollary
\ref{T:center}, or just the fact that $M$ is a finite-dimensional
module over its center.
\end{proof}

\begin{lemma} The categories of abelian
tracial von Neumann algebras and probability measure algebras are
equivalent.
\end{lemma}

\begin{proof}
Given a tracial von Neumann algebra $(M,\tau)$ define the
probability measure algebra $\cB(M,\tau)$ to be the complete Boolean
algebra of projections of $M$, with $\mu(p)=\tau(p)$ giving a
probability measure on $\cB$. Since $\tau$ is normal, $\mu$ is
$\sigma$-additive.

Given a probability measure algebra $(\cB,\mu)$ define the tracial
von Neumann
algebra $M(\cB,\mu)$ to be the associated $L^\infty$ algebra, with $\tau$ being integration against $\mu$.  


It is clear that $\cB(M(\cB,\mu))$ is isomorphic to  $(\cB,\mu)$ and
that $M(\cB(M,\tau))$ is  isomorphic to $(M,\tau)$. It is also clear
that these two operations agree with morphisms, so we have an
equivalence of categories.
\end{proof}

The following is a consequence of Maharam's theorem and we shall
need only the case when $\cB$ is countably generated.

\begin{prop} \label{T:prob}
If $(\cB,\nu)$ is a  probability measure algebra then all of its
ultrapowers associated with nonprincipal ultrafilters on $\bbN$ are
isomorphic.
\end{prop}

\begin{proof}
Recall that the \emph{Maharam character} of a measure algebra is the
minimal cardinality of a set that generates the algebra and that a
measure algebra is \emph{Maharam homogeneous} if for every nonzero
element $b$ the Maharam character of the algebra restricted to $b$
is equal to the Maharam character of the whole algebra. Maharam's
theorem
 (\cite[331L]{Fr:MT3}) implies that a Maharam-homogeneous probability
 measure algebra of
 Maharam character $\kappa$ is isomorphic to the Haar measure algebra on $2^\kappa$.

 Let us first give a proof in the case when $\cB$ is Maharam homogeneous.
By Maharam's theorem it will suffice to show $\cB^{\cU}$ is Maharam
homogeneous of character $\kappa^{\aleph_0}$ for every  nonprincipal
ultrafilter $\cU$ on $\bbN$. By $\kappa^{<\bbN}$ we denote the set
of all finite sequences of ordinals less than $\kappa$. Since
 $\cB$ is isomorphic to the Haar measure algebra on $2^\kappa$ and  the
cardinality of $\kappa^{<\bbN}$ is equal to $\kappa$, we can  fix a
stochastically independent sequence $x_s$, for $s\in
\kappa^{<\bbN}$. Hence
 $\nu(x_s)=1/2$ for all $s$ and $\nu(x_s\cap x_t)=1/4$ whenever
$s\neq t$.

For $f\colon \bbN\to \kappa$  define $\bfx_f$ by its representing
sequence $x_{f\rs n}$, for $n\in \bbN$.
 For $f\neq g$ and for all large enough $n$ we have $\nu(x_{f\rs n}\cap x_{g\rs n})=1/4$.
Therefore $\bfx_f$, for $f\colon \bbN\to \kappa$, is a
stochastically independent family of size $\kappa^{\aleph_0}$. Since
the cardinality of $\cB^{\cU}$ is $\kappa^{\aleph_0}$ we conclude it
is a Maharam-homogeneous algebra of Maharam character
$\kappa^{\aleph_0}$.
  Using  Maharam's theorem again we conclude  all such ultrapowers are isomorphic.

 In the general case, we can write $\cB$ as a direct sum of denumerably many
 Maharam-homogeneous algebras. We  assume this sequence is infinite since
 the finite case is slightly easier.
Write  $\cB$ as a direct sum $\bigoplus_{i=0}^\infty \cB_i$ where
each $\cB_i$ is Maharam homogeneous. It is clear from
Lemma~\ref{T:dsum} that $\cB^{\cU}$ is isomorphic to
$\bigoplus_{i=0}^\infty \cB_i^{\cU}$ for any ultrafilter $\cU$ and
the conclusion follows by the first part of the proof.
 \end{proof}

As a side remark, we note that by a result from \cite{FaHaSh:Model2}
the above implies that the theory of probability measure algebras
does not have the order property, and thus gives a different proof of
the same result from \cite{BYBHU} (see also \cite{BY:CAT}).

\begin{theorem} \label{T:typeI}
Assume the Continuum Hypothesis fails. A separable tracial von
Neumann algebra has the property that all of its ultrapowers
associated with nonprincipal ultrafilters on~$\bbN$ are isomorphic
if and only if it is type I.
\end{theorem}

\begin{proof}
It follows from Lemmas \ref{T:dsum} and \ref{T:typeIn} and
Proposition \ref{T:prob} that tracial type I algebras do not have
the order property.

Note that a tracial $\text{II}_1$ algebra (even a non-factor)
unitally contains the sequence $\mathbb{M}_{2^n}$, so has the order
property, and therefore has non-isomorphic ultrapowers.  A tracial
von Neumann algebra which is not of type I must have a type
$\text{II}_1$ summand, so it has non-isomorphic ultrapowers by Lemma
\ref{T:dsum} and Proposition \ref{P3}.
\end{proof}

We are now in the position to handle the relative commutants of
separable II$_1$ factors. Recall that a separable II$_1$ factor is said to be
{\it McDuff} if it is isomorphic to its tensor product with the
hyperfinite II$_1$ factor. By McDuff's theorem from
\cite{McDuff:Central}, being McDuff is equivalent to having
non-abelian relative commutant.  The following Theorem says that a
II$_1$ factor is McDuff iff its relative commutant has the order
property.

\begin{theorem}\label{T.rel.comm}
Suppose that $\CA$ is a separable II$_1$ factor.  If the Continuum
Hypothesis fails then the following are equivalent:
\begin{enumerate}
\item there are
nonprincipal ultrafilters $\cU$ and $\cV$ on $\bbN$ such that the
relative commutants $A'\cap A^{\cU}$ and $A'\cap A^{\cV}$ are not
isomorphic.
\item there are
nonprincipal ultrafilters $\cU$ and $\cV$ on $\bbN$ such that the
unitary groups of the relative commutants $A'\cap A^{\cU}$ and
$A'\cap A^{\cV}$ are not isomorphic.
\item some (all) relative
commutant(s) of $\CA$ are type II$_1$.
\end{enumerate}

If the Continuum Hypothesis holds then all the relative commutants
of $\CA$ and all their unitary groups are isomorphic.
\end{theorem}

\begin{proof}
As in the proof of Theorem~\ref{T.type-II-1}, the equivalence of the
first two clauses of the first paragraph follows from Dye's result
(\cite{Dye:On}) that two von Neumann algebras are isomorphic if and
only if their unitary groups are isomorphic. The last sentence in the Theorem follows from \cite{GeHa}. McDuff showed
(\cite{McDuff:Central}) that for nonprincipal ultrafilters $\cU$ and
$\cV$ on $\bbN$, the relative commutants $A'\cap A^{\cU}$ and
$A'\cap A^{\cV}$ are  always
either
\begin{enumerate}
\item $\mathbb{C}$,
\item abelian, non-atomic and of density character $\fc$, or
\item of type II$_1$.
\end{enumerate}

In the first case, clearly all relative commutants are isomorphic.
In the second case, as in the proof of Theorem~\ref{T:prob}, one
can show that the relative commutant is in fact isomorphic to
$L^\infty(\cB)$ for a Maharam-homogeneous measure algebra $\cB$ of
character density $\fc$.  It follows then by Maharam's theorem that
all relative commutants in this case are isomorphic. We are left
then with the possibility that a relative commutant of $\CA$ is type
II$_1$. By the previous theorem, this means that that relative
commutant has the order property.  The proof then follows from
Proposition \ref{P4}.
\end{proof}

\section{C*-algebras}
In this subsection we consider norm ultrapowers of C*-algebras and
their unitary groups. Here $U(A)$ denotes the unitary group of a
unital C*-algebra~$A$. The following is the analogue (although
simpler) of Theorems \ref{T.type-II-1} and \ref{T.rel.comm} in the
C*-algebra case.

\begin{theorem} \label{T.C*}
For an  infinite-dimensional  separable unital C*-algebra $A$ the
following are equivalent.
\begin{enumerate}
\item \label{T.C*.1} For all nonprincipal ultrafilters $\cU$ and $\cV$ the
ultrapowers $A^{\cU}$ and $A^{\cV}$ are isomorphic.
\item \label{T.C*.2} For all nonprincipal ultrafilters $\cU$ and $\cV$ the
relative commutants of $A$ in the ultrapowers $A^{\cU}$ and
$A^{\cV}$ are isomorphic.
\item \label{T.C*.3}For all nonprincipal ultrafilters $\cU$ and $\cV$ the
ultrapowers $U(A)^{\cU}$ and $U(A)^{\cV}$ are isomorphic.
\item \label{T.C*.3b}For all nonprincipal ultrafilters $\cU$ and $\cV$ the
relative commutants of $U(A)$ in the ultrapowers $U(A)^{\cU}$ and
$U(A)^{\cV}$ are isomorphic.
\item \label{T.C*.4}The Continuum Hypothesis holds.
\end{enumerate}
If $A$ is separable but not unital then \eqref{T.C*.1},
\eqref{T.C*.2} and \eqref{T.C*.4} are equivalent.
\end{theorem}

Some instances  of Theorem~\ref{T.C*} were proved in \cite{GeHa} and
\cite{Fa:Relative}. If the separability of $A$ is not assumed then
\eqref{T.C*.1} and \eqref{T.C*.2} in Theorem~\ref{T.C*} need not be
equivalent (see \cite{FaPhiSte:Relative}).

 For a unital
C*-algebra~$A$ let $U(A)$ denote its unitary group equipped with the
norm metric. The analogue of Dye's rigidity result used in the
proofs of Theorems \ref{T.type-II-1} and \ref{T.rel.comm} for
unitary groups of C*-algebras used above is false (Pestov, see
\cite{G-student} also \cite{GaRo:Characterizing}) although this is
true for simple AF C*-algebras by~\cite{G-student}.

A C*-algebra $A$ has the  \emph{order property with respect to $g$}
if for every
$\e>0$ there are arbitrarily long finite $\prec_{g,\e}$-chains (see
\S\ref{S.Setup}). We say that the \emph{relative commutant type of
$A$ has the order property with respect to $g$} if there are  $n$, a
*-polynomial $P(x_1,\dots, x_n, y_1,\dots, y_n)$ in $2n$ variables
such that with $g(\vec x,\vec y)=\|P(\vec x,\vec y)\|$
  for  every finite $F\subseteq A$,  
  and every $m\in \bbN$,  there is a
$g$-$1/m$-chain of length $m$ in $A_{\leq 1}^n$ all of whose
elements $1/m$-commute with all elements of~$F$.

\begin{lemma}
If $A$ is an infinite-dimensional separable unital C*-algebra and
$\CU$ is a non-principal ultrafilter on $\bbN$ then $A' \cap
A^{\CU}$ is infinite-dimensional and in fact, non-separable.
\end{lemma}

\begin{proof}  We divide this proof into cases depending on whether or not the $C^*$-algebra is continuous trace. 
Recall that a $C^*$-algebra $A$ is said to have continuous trace if its spectrum $T$ is Hausdorff and is locally
Morita equivalent to $C_0(T)$.
We will show that an infinite-dimensional unital
continuous trace $C^*$-algebra $A$ must have infinite-dimensional
center.  We could not find this explicitly stated in the literature,
so we give a short proof. Suppose $A$ is continuous trace; it
is type $\text{I}$ and its spectrum $\hat{A}$ is Hausdorff
(\cite[IV.1.4.16]{Black:Operator}). By the Dauns-Hofmann theorem,
$\mathcal{Z}(A) \simeq C(\hat{A})$. If $\mathcal{Z}(A)$ were
finite-dimensional, then $\hat{A}$ would be finite and discrete; by
\cite[X.10.10.6a]{D:book} $A$ would be a direct sum of simple unital
type I algebras.  This would force $A$ to be finite-dimensional, a
contradiction. Always $\mathcal{Z}(A) \subseteq A' \cap A^\CU$, so
if $A$ is continuous trace, the lemma follows.

In the remainder of the proof we suppose $A$ is not continuous
trace.  By \cite[Theorem 2.4]{AkPed}, $A$ has a nontrivial central
sequence $a(n)$ -- but note that these authors work with limits at
infinity, not an ultrafilter.  Passing to a subsequence if
necessary, we may assume that there is $c > 0$ such that for all
$n$,
$$\inf_{z \in \mathcal{Z}(A)} \|a(n) - z\| > c.$$
Now moving to the ultrapower, it is clear that $a(n)$ belongs to $A'
\cap A^\CU$ but is not equivalent to a constant sequence.

We have $\mathbf{a} = a(n) \in (A' \cap A^{\CU}) \setminus
A$. Let $\epsilon = d(\mathbf{a},A)/2$; $\epsilon > 0$ since $A$ is
complete.
We may assume $\epsilon\leq c$. 
Note that for every finite subset $G$ of the sequence $\{a(n)\}$ the set 
$\{j : \| a - a(j)\| > \epsilon$  for all $a\in G\}$ is in $U$.  Let $F_n$, for $n\in N$, be an increasing
sequence of finite subsets of $A$ with dense union. For every $m$ and every $\delta>0$ the set 
$\{j: \| [b,a(j)] \| <\delta$ for all $b\in F_m\}$ is in $U$. We can therefore recursively find disjoint 
sets $G_n$ for $n$ in $N$ such that  for all $n$

\begin{enumerate}
\item  $\| a(j) - a(k) \|\geq  \epsilon$ for all distinct $j$ and $k$ in $G_n$
\item For every $m\leq n$ and every $b\in F_m$ and $j\in G_n$, $\| [a(j), b ] \| <1/m$.  
\item $| G_n | =2^n$ and $G_n$ is enumerated as $a(s)$ for $s\in s^n$. 
\end{enumerate}
If $x\in 2^{\mathbb N}$ then (with $x| n$ denoting the first $n$ digits of $x$)
the sequence $a(x|n)$ for $n\in N$ is central. Let ${\mathbf a}(x)$ denote
the element of $A^U$ corresponding to it. By (2) it is in $(A'\cap A^U)\setminus A$.
If $x\neq y$ are in $2^{\mathbb N}$ then $x|n\neq y|n$ for a large $n$ and therefore  
(1) implies $\| {\mathbf a}(x) - {\mathbf a}(y)\|\geq \epsilon$.  

We have therefore proved that the relative commutant of $A$ in $A^U$ is nonseparable. 
\end{proof}

\begin{lemma} \label{L.B3} Assume $A$ is an infinite-dimensional separable unital C*-algebra. Then
both~$A$ and its relative commutant type have the order property.
\end{lemma}

\begin{proof}
Both A and any of its relative commutants are infinite-dimensional, by the previous lemma, so their maximal abelian *-subalgebras are also infinite-dimensional (\cite{O:findim}).
In order to prove the
lemma then, it suffices to consider the case when $A$ is abelian. By the
Gelfand transform $A$ is isomorphic to $C_0(X)$ for an infinite
locally compact (possibly compact) Hausdorff space $X$. In $X$ find
a sequence of distinct $x_n$ that converges to some $x$ (with $x$
possibly in the compactification of $X$). For each $n$ find a
positive $a_n\in C_0(X)$ of norm 1 such that $a_n(x_i)=1$ for $i\leq
n$ and $a_n(x_i)=0$ for $i>n$. By replacing $a_n$ with the maximum
of $a_j$ for $j\leq n$ we may assume this is an increasing sequence.
Moreover, we may assume the support $K_n$ of each $a_n$ is compact
and $a_m$ is identically equal to 1 on $K_n$ for $n<m$. Hence $a_m
a_n=a_n$ if $n\leq m$.

Now consider the formula $\varphi(x,y)=\|xy-y\|$. Then
$\varphi(a_m,a_n)=0$ if $m>n$ and $\varphi(a_m,a_n)=1$ if $m<n$, and
the order property for $A$ follows.
\end{proof}

In connection with the first paragraph of the proof of
Lemma~\ref{L.B3} it is worth mentioning that there is a nonseparable
C*-algebra all of whose masas are separable (\cite{Po:Orthogonal}).

The idea for the following is well-known. Compare for example with
\cite[Example 8.2]{Iovino:StableI}. 

\begin{lemma} \label{unitary-lemma} For
every   infinite-dimensional unital C*-algebra both its  unitary
group
 and the  relative commutant type of its unitary group
have the order property. \end{lemma}

\begin{proof} Let $A$ be an infinite-dimensional unital C*-algebra. Since it
contains an infinite-dimensional abelian C*-algebra we may assume
$A=C(X)$ for an infinite compact Hausdorff space $X$. Therefore
$U(A)$ is isomorphic to $C(X,\bbT)$, where $\bbT$ denotes the unit
circle. Let $x_n$, for $n\in \bbN$, be a nontrivial convergent
sequence in $X$ and let $U_n\ni x_n$ be disjoint open sets. Find $0
\leq f_n \leq 1/4$ in $C(X)$ such that $f_n(x_n)=1/4$ and
$\supp(f_n)\subseteq U_n$. Let  $a_n=\exp(2\pi i \sum_{j\leq n}f_j)$
and $b_n=\exp(2\pi i f_n)$. Then
\[
g_0(a,b,a',b')=\|1-ab'\|
\]
 is such that $g_0(a_m,b_m,a_n,b_n)$ is equal to $|1-i|$ if $m<n$ and
to $2$ if $m\geq n$. It is now easy to modify $g_0$ to $g$ as
required.

The result for the relative commutant type follows as in
Lemma~\ref{L.B3}.
 \end{proof}



\begin{proof}[Proof of Theorem~\ref{T.C*}]
 The implications from \eqref{T.C*.4} were proved in \cite{GeHa}
 and the converse implications follow by
Lemma~\ref{L.B3}, Lemma~\ref{unitary-lemma}, and
Proposition~\ref{P4.C*}.
\end{proof}

\noindent{\it Added June, 2010:} Our results give
 only as many nonisomorphic ultrapowers as there are uncountable cardinals
$\leq \fc$ (i.e., at least two). This inspired \cite{FaSh:Dichotomy}, where
it was proved that every separable metric structure $A$ either has
all of its ultrapowers by ultrafilters on $\bbN$ isomorphic or it
has $2^{\fc}$ nonisomorphic such ultrapowers (see \cite[Corollary
4]{FaSh:Dichotomy}). The analogous result for relative commutants of
C*-algebras and II$_1$ factors is given in
\cite[Proposition~8.4]{FaSh:Dichotomy}. The answer to Popa's
question, mentioned before, can be found in \cite[Proposition
8.3]{FaSh:Dichotomy}.

\bibliographystyle{amsplain}
\bibliography{stablebib}
\end{document}